\documentclass{amsart}
\usepackage{graphicx}
\usepackage{amsfonts}
\usepackage{amssymb}
\usepackage{amsmath}
\usepackage{amsthm}
\usepackage{epsfig}
\usepackage{youngtab}
\usepackage[all]{xy}
\usepackage[english]{babel}

\newcommand{\mR}{\mathbb{R}}
\newcommand{\mC}{\mathbb{C}}
\newcommand{\mcR}{\mathcal{R}}
\newcommand{\mcS}{\mathcal{S}}
\newcommand{\mcM}{\mathcal{M}}

\newcommand{\mcC}{\mathcal{C}}

\newcommand{\mcP}{\mathcal{P}}

\newcommand{\mcV}{\mathcal{V}}
\newcommand{\mcW}{\mathcal{W}}

\newcommand{\mN}{\mathbb{N}}

\newcommand{\mS}{\mathbb{S}}
\newcommand{\mV}{\mathbb{V}}

\newcommand{\ux}{\underline{x}}

\newcommand{\up}{\underline{\partial}}

\newcommand{\uu}{\underline{u}}

\newcommand{\eps}{\epsilon}

\newcommand{\so}{\mathfrak{so}}

\DeclareMathOperator{\Ker}{Ker}

\DeclareMathOperator{\Spin}{Spin}

\newcommand{\DD}[2]{D_{#1,#2}}
\newcommand{\TW}[2]{T_{#1,#2}}
\newcommand{\PA}[2]{G_{#1,#2}}
\newcommand{\RS}[1]{R_{#1}}
\newcommand{\LA}[1]{\Delta_{#1}}

\newcommand{\LE}[2]{|{#1},{#2}|}
\newcommand{\LAg}{\mathfrak{g}}
\newcommand{\LAh}{\mathfrak{h}}

\theoremstyle{plain}

\newtheorem{theorem}{Theorem}
\newtheorem{lemma}{Lemma}
\newtheorem{corollary}{Corollary}
\newtheorem{definition}{Definition}

\theoremstyle{remark}

\begin{document}
\bibliographystyle{abbrv}
\title{Factorization of Laplace operators on higher spin representations}

\author{David Eelbode} 
\address{David Eelbode \\ 
Department of Mathematics and Computer Science \\
University of Antwerp \\
Middelheimlaan 1, 2020 Antwerpen, Belgium}
\email{David.Eelbode@ua.ac.be}
\author{
Dalibor \v{S}m\'{\i}d}
\address{Dalibor \v{S}m\'{\i}d \\ Mathematical Institute \\ Charles University Prague \\ Sokolovska 83, 18000 Praha 8, Czech Republic}
\email{smid@karlin.mff.cuni.cz}

\date{\today}

\begin{abstract}
This paper deals with the problem of factorizing integer powers of the Laplace operator acting on functions taking values in higher spin representations. This is a far-reaching generalization of the well-known fact that the square of the Dirac operator is equal to the Laplace operator. Using algebraic properties of projections of Stein-Weiss gradients, i.e. generalized Rarita-Schwinger and twistor operators, we give a sharp upper bound on the order of polyharmonicity for functions with values in a given representation with half-integral highest weight.
\end{abstract}

\keywords{Clifford algebras, Dirac operators, Rarita-Schwinger operator, Laplace operator}

\subjclass[2010]{30G35, 22E46, 31A30}

\thanks{The second author was supported by the grant GA 201/08/0397 and the research plan MSM 0021620839 of the Ministry of Education of the Czech Republic.}

\maketitle

\section{Introduction}

Clifford analysis provides a generalization of complex analysis in the plane to a higher-dimensional setting, in which the role of the Cauchy-Riemann operator is played by the Dirac operator. It is centered around the study of functions on the vector space $\mR^m$ taking values in its Clifford algebra or the corresponding spinor representation, see e.g. \cite{Brackx1982,Delanghe1992} for the standard references. Let $\{e_1,\ldots,e_m\}$ denote an orthonormal basis for the Euclidean vector space $\mR^m$. The Clifford algebra $\mC_m$ of the complexified vector space $\mC^m$ is generated by the multiplicative relations $e_ie_j+e_je_i=-2\delta_{ij}$ and has the structure of a graded vector space $\mC_m=\bigoplus_{k=0}^m \mC_m^{(k)}$, where $\mC_m^{(k)}$ is spanned by elements $e_{i_1}\ldots e_{i_k}$ with $1 \leq i_1 < \ldots < i_k \leq m$. The space $\mC_m^{(2)}$, provided with the commutator bracket, defines a model for the Lie algebra $\so (m)$ of the group $\Spin (m)$, which can be defined as the Lie group consisting of all even Clifford products of unit vectors in $\mR^m$. The spinor representation $\mS$ of $\Spin(m)$ can be realized as a minimal left ideal in $\mC_m$ by means of a standard construction involving the Witt basis. For $m$ odd, $\mS$ is irreducible with highest weight $\left(\frac{1}{2},\ldots,\frac{1}{2}\right)$, for $m$ even it decomposes into two irreducible summands $\mS^+ \oplus \mS^-$ with weights $\left(\frac{1}{2},\ldots\pm\frac{1}{2}\right)$. For the sake of convenience, we will restrict to the case of odd dimension in this paper (although it should be pointed out that the results also apply to the case of even dimensions, taking parity changes into account). \\
\noindent
The Dirac operator $\up_x:=\sum_{i=1}^m e_i \partial_{x_i}$ acts on the space $\mcC^{\infty}(\mR^m,\mS)$ of $\mS$-valued functions as a $\Spin(m)$-invariant elliptic differential operator, and its kernel is defined as the space of monogenic functions. In view of the fact that $\up_x^2=-\Delta_m$, the Dirac operator factorizes the Laplace operator in $m$ dimensions, which means that monogenic functions are also harmonic.\\
\noindent
In recent years, several authors \cite{Burevs2000,Burevs2001,Burevs2001a,VanLancker2006,Brackx,Eelbode2009,Eelbode2009a} have been working on generalizations of Clifford analysis techniques to the so-called higher spin representations: these are arbitrary irreducible representations $\mS_\lambda$ of $\Spin(m)$ with dominant half-integral highest weight 
$$ \lambda' :=\lambda+\left(\frac{1}{2},\ldots,\frac{1}{2}\right)\ ,$$ 
where $\lambda$ contains integers only. Given this highest weight, there is (up to normalization) a unique invariant first-order differential operator $$\mcR_{\lambda}:\mcC^{\infty}(\mR^m,\mS_\lambda) \rightarrow \mcC^{\infty}(\mR^m,\mS_\lambda)\ .$$
It is given as the projection of the Stein-Weiss gradient of a function to the irreducible summand $\mS_\lambda$ in $(\mR^m)^* \otimes \mS_\lambda$. For $\lambda=(1,0,\ldots,0):=(1)$ the operator $\mcR_{(1)}$ is the multidimensional analogue of the classical Rarita-Schwinger operator, see [RS]. In full generality, the operator $\mcR_{\lambda}$ is called the higher spin Dirac (HSD) operator for the representation $\mS_\lambda$ of $\Spin(m)$ or its Lie algebra $B_n$, see Section 3. \\
\noindent
The main result of this paper is a formula for the factorization of a suitable power of the Laplace operator on an arbitrary higher spin representation through HSD operators:
\begin{theorem}
Let $\lambda$ be a dominant integral weight, and let $n \in \mN$ be an integer satisfying $n > \lambda_1$. Then there exists an invariant differential operator $A_{\lambda,n}$ acting on $\mcC^{\infty}(\mR^m,\mS_\lambda)$ such that
$$\Delta^n = \mcR_{\lambda} A_{\lambda,n} \mcR_{\lambda}\ .$$
\end{theorem} 
\noindent
Thus every function in $\Ker \mcR_{\lambda}$ is $(\lambda_1+1)$-polyharmonic. The structure of $\Ker \mcR_{\lambda}$ is explicitly known only for the cases $\lambda=(k)$, see e.g. \cite{Burevs2001}, and $\lambda=(k,l)$, see e.g. \cite{Brackx}. In both cases the bound $(\lambda_1 + 1)$ is optimal, i.e. there are functions in $\Ker \mcR_{\lambda}$ which are not $\lambda_1$-polyharmonic.\\
\noindent
Our paper is structured as follows: we will briefly explain how the HSD operators arise within Clifford analysis framework in Section 2. At the same time, we will show how knowledge on the factorization of the Laplace operator on the underlying higher spin representations can be used to describe the (polynomial) homogeneous kernel spaces. Sections 3 and 4 are then devoted to the actual factorization of the Laplace operator. To do this, we will make use of techniques coming from representation theory. Therefore, we will briefly fix the notations in Section 3, before turning to the main problem in Section 4.


\section{Higher spin Clifford analysis}
Let $\mC_m$ be the universal complex Clifford algebra generated by an orthonormal basis $\{e_p : 1 \leq p \leq m\}$ for the vector space $\mR^m$, satisfying the multiplicative relations $e_pe_q + e_qe_p = -2\delta_{pq}$. As was pointed out earlier, the spinor space $\mS$ can then be realized as a minimal left ideal $\mS = \mC_m I$, where $I$ denotes a primitive idempotent (see \cite{Delanghe1992}).\\
\noindent 
The main motivation for considering HSD operators within the framework of Clifford analysis comes from the crucial fact that any higher spin representation $\mS_\lambda$ can be explicitly realized as a vector space of $\mS$-valued polynomials in several vector variables $\uu_p \in \mR^m$, satisfying certain systems of equations. In what follows, the variables $\uu_p$ and their associated (classical) Dirac operators $\up_p := \up_{u_p}$ have to be seen as dummy variables, ``fixing the values''.
\begin{definition}\label{simplicial monogenic}
A function $ f:\mR^{km}\rightarrow \mS,\,\, (\uu_1,\ldots,\uu_k)\mapsto f(\uu_1,\ldots,\uu_k)\ $ is simplicial monogenic iff the following conditions are satisfied:
$$\begin{array}{ccccl}
\up_p f & = & 0 && 1 \leq p \leq k\\
\langle\uu_p,\up_{q}\rangle f & = & 0 && 1\leq p < q \leq k
\end{array}\ .$$
The vector space of $\, \mS$-valued simplicial monogenic polynomials which are homogeneous of order $\lambda_p$ in the variable $\uu_p$ will be denoted by $\mS_\lambda = \mS_{\lambda_1,\ldots,\lambda_k}$ (with $\lambda_1 \geq \ldots\geq \lambda_k \geq 0$  from now on). 
\end{definition}
\noindent
In [CSV], the following result was obtained: 
\begin{theorem}
With respect to the regular representation $L$ of the spin group on $\mS_\lambda$, given by $L(s)f(\uu_1,\ldots,\uu_k) = s f(\overline{s}\uu_1 s,\ldots,\overline{s}\uu_k s)$, the space $\mS_\lambda$ defines a model for the the irreducible Spin$(m)$-representation with highest weight $\lambda'$.
\end{theorem}
\noindent
This means that a HSD operator in Clifford analysis can be seen as an operator $\mcR_\lambda$ acting on functions $f(\ux;\uu_1,\ldots,\uu_k) \in \mcS_\lambda$, such that 
$$ \up_p\big(\mcR_\lambda f(\ux;\uu_p)\big) = 0 = \langle\uu_p,\up_q\rangle\big(\mcR_\lambda f(\ux;\uu_p)\big)\ , $$
with $p$ and $q$ indices as in the definition. Let us then first consider the (special) cases mentioned earlier, i.e. for which the kernel spaces are fairly understood. 

\subsection{The Rarita-Schwinger operators $\mcR_k$}

Fixing a highest weight $\lambda' = (k)'$, it follows from the Theorem above that the irreducible module with highest weight $\lambda'$ can be realized as the space $\mcM_k := \mcP_k(\mR^m,\mS) \cap \ker\up_u$ of $k$-homogeneous monogenics. The corresponding HSD operator, generalizing the Rarita-Schwinger operator coming from physics \cite{Rarita1941}, is usually denoted as $\mcR_k$ and is then defined as
$$ \mcR_k := \left(1 + \frac{\uu\:\up_u}{2k+m-2}\right)\up_x\ . $$
The function theory for this operator was to a great extent developed in \cite{Burevs2000,Burevs2001}, in which the fundamental solution, the Cauchy integral formula and explicit descriptions for polynomial solutions were obtained. In particular the following result was obtained: 
\begin{theorem}
There exists a differential operator $A_{2k+1}$ acting between spaces of $\mcM_k$-valued functions, such that $\mcR_kA_{2k+1} = \Delta^{k+1}_m$. 
\end{theorem}
\noindent
In order to illustrate the link between this result and the structure of the space $\ker_h\mcR_k$ of $h$-homogeneous (polynomial) solutions for the operator $\mcR_k$, we also mention the following crucial result:
\begin{theorem}
The induction principle states that 
$$ \ker_h\mcR_k = \mcM_{h,k} \oplus \up_x^{-1}\big(\uu\:\ker_{h-1}\mcR_{k-1}\big)\ . $$
Here, the $\mcM_{h,k}$ stands for the space of $(h,k)$-homogeneous polynomials which are monogenic in both variables $(\ux,\uu)$, and $\up_x^{-1}$ associates to each $g \in \ker_{h-1}\mcR_{k-1}$ the unique solution of the system
$$ 
\begin{array}{ccl}
\up_x f & = & \uu g\\
\up_u f & = & 0
\end{array}
$$
\end{theorem}
\noindent
\textbf{Remarks:}
\begin{enumerate}
\item This theorem exhibits an important feature of HSD operators, as opposed to the classical Dirac operator: the space of polynomial solutions is no longer irreducible under the action of the spin group. As a matter of fact, decomposing this kernel space is a central problem within this setting.
\item The action of the operator $\up_x^{-1}$ from the Theorem is nothing but the action of a (dual) twistor operator, cfr. infra. 
\item In particular, the theorem above illustrates that solutions for $\mcR_k$ are indeed polyharmonic: it suffices to perform induction on the index $k$ to see that $\Delta_m^{k+1}$ indeed acts trivially on $\ker\mcR_k$. 
\end{enumerate}
The Theorem essentially says that in order to decompose the kernel of the operator $\mcR_k$, one needs two basic ingredients: a special type of solutions (the double monogenics), together with information on the action of the twistor operator on solutions for the ``simpler'' system $\mcR_{k-1}f = 0$.

\subsection{The higher spin operators $\mcR_{k,l}$}

In a recent series of papers \cite{Eelbode2009a,Eelbode2009}, we have considered the case where $\lambda' = (k,l)'$. Our motivation for doing so lies in the fact that the case of the Rarita-Schwinger operator, does not yet contain the germs from which the behaviour of the most general case can be conjectured. The corresponding HSD operator $\mcR_{k,l}$, acting on the space of $\mcS_{k,l}$-valued functions taking, was found to be given by:
$$ \mcR_{k,l} := \left(1 + \frac{\uu_1\:\up_1}{2k+m-2}\right)\left(1 + \frac{\uu_2\:\up_2}{2l+m-4}\right)\up_x\ .$$
The main function theoretical results were developed, and in particular an accurate description of the polynomial solution space was obtained. Without going into full details, we mention the following result:
\begin{theorem}
A refined version of the induction principle leads to the following decomposition for the kernel of the operator $\mcR_{k,l}$: 
$$ \ker_h\mcR_{k,l} \cong \bigoplus_{i=0}^l\bigoplus_{j=0}^{k-l}\mcM_{h,k-j,l-i}^s\ . $$
\end{theorem}
For $(i,j) = (0,0)$, one obtains the space $\mcM_{h,k,l}^s$, which is the analogue of the space of double monogenics in the present setting. To be more precise, this is the space of triple monogenics, satisfying an additional constraint (hence the extra superscript) to make sure that the correct values are assumed. For all other choices $(i,j) \neq (0,0)$, one obtains solutions for ``simpler'' equations $\mcR_{k-i,l-j}f = 0$ which are in certain sense inverted through a suitable action of associated twistor operators. It is then crucial to point out that \textit{not all} the indices $(k-i,l-j)$ are used in the Theorem above. As a matter of fact, the indices needed in the decomposition above are \textit{precisely} the indices needed in the decomposition of the Laplace operator on $\mcS_{k,l}$-valued functions (see Section 4 for details).

\subsection{The general case $\mcR_\lambda$}

In a recent paper \cite{DeSchepper2010}, techniques were developed to obtain explicit expressions for general HSD operators $\mcR_\lambda$. Moreover, the authors obtained the corresponding special class of solutions $\mcM_{h,\lambda}^s$ for these operators which generalizes the space of double monogenics from the Rarita-Schwinger case. To complete the description of the polynomial kernel space, one then only needs to understand the behaviour of the twistor inversion, relating solutions for $\mcR_\lambda$ to solutions for ``simpler'' HSD operators. A crucial step towards this understanding is the Theorem we are about to prove in Section 4, regarding the decomposition of the Laplace operator. 


\section{Notations}

Let $\LAg$ be a complex semisimple Lie algebra of the type $B_n$ with Cartan subalgebra $\LAh$. By $(\epsilon_i)_1^n$ we denote the standard basis of $\mC^n \equiv \LAh^*$. The dominant part $\Lambda_d$ of the weight lattice $\Lambda$ of $\LAg$ can be expressed in the standard basis as $\sum_1^n \lambda_i \epsilon_i$ with $\lambda_i$ either all integral or all half-integral non-negative numbers satisfying $\lambda_i\geq\lambda_{i+1}$ for all $i\in\{1,\ldots,n-1\}$. For $\lambda$ integral, this is the highest weight of a tensorial representation with symmetry given by the Young diagram with $\lambda_i$ boxes in the $i$-th row. Let us denote by $\lambda'$ the weight $\lambda + \left(\frac{1}{2}\right)_n$, i.e. the Cartan product of $\lambda$ and the spinor representation. Finally, we denote by $\mV_\lambda$ (resp. $\mS_\lambda$) the irreducible representation with highest weight $\lambda$ (resp. $\lambda'$) and by $\mcV_\lambda$ (resp. $\mcS_\lambda$) the set of smooth functions with values in this representation.

Next, we also define a graph $\mcW$ whose set of nodes consists of all integral dominant weights and where two weights $\lambda,\mu$ are joined by an edge iff $\sum_{i=1}^n (\lambda_i-\mu_i)^2 = 1$. The graph $\tilde{\mcW}$ has the same set of nodes, but the edges are defined by $\sum_{i=1}^n (\lambda_i-\mu_i)^2 \leq 1$. It is then easily seen that there exists a first-order $\LAg$-invariant operator from $\mcV_\lambda$ to $\mcV_\mu$ iff $\lambda$ and $\mu$ are joined in $\tilde{\mcW}$, and we denote such operators by $\DD{\lambda}{\mu}:\mcV_\mu \rightarrow \mcV_\lambda$. Each of these operators is unique, up to a constant: the issue of normalization in the cases of interest will be dealt with later. As we are mainly interested in operators between half-integral representations, generalizing twistor and Rarita-Schwinger operators, we introduce two more symbols: 
$$ \TW{\lambda}{\mu}:=\DD{\lambda'}{\mu'}\ \ \mbox{and}\ \ \mcR_{\lambda}:=\DD{\lambda'}{\lambda'}\ .$$ We also introduce a similar symbol $\LA{\lambda}$ for the Laplace operator on $\mcS_{\lambda}$, and for the sake of convenience, we put any operator $\TW{\lambda}{\mu}, \RS{\lambda}, \LA{\lambda}$ for which the weights are non-dominant equal to zero.


\section{Path operators}

First of all, we define two particular sets of weights: 
\begin{definition}
Let $\mu$ and $\nu$ be two integral dominant weights such that $\mu \succ \nu$ in the Bruhat order, i.e. $\mu_i \geq \nu_i$ for all $i\in\{1,\ldots,n\}$. A {\it path} from $\nu$ to $\mu$ is a sequence $(\lambda_p)_0^N$ of weights in $\mcW$ such that $\lambda_0=\mu$, $\lambda_N=\nu$, and for all indices $p\in\{1,\ldots,N\}$, we have that $\lambda_p \succ \lambda_{p-1}$ and $|\lambda_p,\lambda_{p-1}|=1$. This means that for two successive weights $\lambda_{p-1}$ and $\lambda_{p}$, there is an index $Ch(p)\in\{1,\ldots,n\}$ such that $\lambda_p-\lambda_{p-1}=\epsilon_{Ch(p)}$, we call this index the {\it change} at $p$. A reverse path is a sequence of weights $(\lambda_p)_0^N$ such that $(\lambda_p)_N^0$ is a path. When there will be no danger of confusion, we will use the term path for both paths and reverse paths. We call the number of edges $N$ the {\it length} of the path or reverse path and denote it by  $\LE{\mu}{\nu}$, clearly it is the distance between $\mu$ and $\nu$ in the Manhattan metric for any path between $\mu$ and $\nu$. 
\end{definition}
\begin{definition}
Given a dominant integral weight $\mu$, we define {\it the box} $B(\mu)$ of the weight $\mu$ as the set of dominant integral weights $\lambda$ satisfying $\mu_i \geq \lambda_i \geq \mu_{i+1}$ for all $i\in\{1,\ldots,n-1\}$ and $\mu_n \geq \lambda_n \geq 0$. 
\end{definition}
\noindent
Given a path, we then define the following differential operator:
\begin{definition}
For any path $(\lambda_p)_0^N$, the path operator is a composition of twistor operators along the path:
$$
\PA{\mu}{\nu}=\TW{\lambda_{N}}{\lambda_{N-1}}\TW{\lambda_{N-1}}{\lambda_{N-2}} \ldots \TW{\lambda_{1}}{\lambda_0}
$$
Similarly, the operator $\PA{\nu}{\mu}$ is defined by composing along the reverse path. 
\end{definition}
\noindent
The notation $\PA{\mu}{\nu}$ from the definition above is justified by the following crucial lemma:
\begin{lemma}\label{pathops}
Let $\mu \succ \nu$ be dominant integral weights. Then there exists a normalization of the twistor operators $\TW{\lambda}{\omega}$ such that
\begin{enumerate}
\item[(i)] $\PA{\mu}{\nu}$ and $\PA{\nu}{\mu}$ do not depend on the choice of the path from $\nu$ to $\mu$, modulo a constant multiple.
\item[(ii)] For a suitable choice of normalization of the twistor operators, one has that
\begin{align*}
\PA{\nu}{\mu} \mcR_{\mu} &= \mcR_{\nu}\PA{\nu}{\mu} \\
\PA{\mu}{\nu} \mcR_{\nu} &= \mcR_{\mu}\PA{\mu}{\nu}\ .
\end{align*}
\item[(iii)] For all integral dominant weights $\lambda \prec \mu$ not in $B(\mu)$, the path operators $\PA{\mu}{\lambda}$ and $\PA{\lambda}{\mu}$ are zero.
\end{enumerate}
\end{lemma}
\noindent
\textbf{Remark: }note that property $(ii)$ from above was already mentioned in a previous section. It essentially relations solutions for different HSD operators, through the action of a twistor operator.
\begin{proof}
The tool for proving these properties is the decomposition of the twisted Dirac operator. For any dominant (integral) weight $\lambda$, the representation $\mV_\lambda \otimes \mS$ decomposes as $\bigoplus \mV_{\mu}$, the sum being taken over the set $$\Lambda_\lambda:=\left\{\left(\lambda_1+\sigma_1\frac{1}{2},\ldots,\lambda_n+\sigma_n\frac{1}{2}\right)|\sigma_i\in\{\pm1\},i\in\{1,\ldots,n\}\right\},$$
where each summand appears exactly once, up to possible omissions of weights which are not dominant. For a given dominant weight $\lambda$, let us encode summands in this decomposition into a sequence of pluses and minuses like $(-,+,+,-,-,-,+)$. We then consider the operator $D=id \otimes \up_x$ on $\mV_\lambda \otimes \mS$-valued functions, where $\up_x$ is the Dirac operator. This twisted Dirac operator $D$ is then equal to a sum of twistor operators (and HSD operators) on each summand. These are nonzero iff the plus-minus codes of the source and target summands differ at most at one position. If we think of these $2^n$ $(\pm)$-codes as vertices of an $n$-dimensional hypercube, there is a nonzero twistor operator between two different vertices iff they are joined by an edge of the hypercube. Since $D^2=-\Delta$ is a scalar operator which preserves each summand, we have all crucial information encoded in diagrams like this:
$$
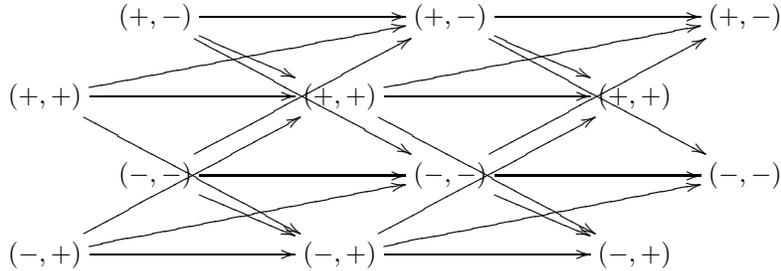
\xymatrix@C=8pt@R=14pt{
                              &(+,-)\ar[rrrr]\ar[drrr]\ar[ddrrrr]& & & &(+,-)\ar[rrrr]\ar[drrr]\ar[ddrrrr]& & &  &(+,-)\\
(+,+)\ar[rrrr]\ar[ddrrrr]\ar[urrrrr]& & & &(+,+)\ar[rrrr]\ar[ddrrrr]\ar[urrrrr]& & & &(+,+)&  \\
                              &(-,-)\ar[rrrr]\ar[drrr]\ar[uurrrr]& & & &(-,-)\ar[rrrr]\ar[drrr]\ar[uurrrr]& & &  &(-,-)\\
(-,+)\ar[rrrr]\ar[uurrrr]\ar[urrrrr]& & & &(-,+)\ar[rrrr]\ar[uurrrr]\ar[urrrrr]& & & &(-,+)&  \\    
}
$$
Each square with nodes $(++),(+-),(-+)$ and $(--)$ hereby represents the $n$-dimensional hypercube of summands and the arrows describe all nonzero twistor (and HSD) operators in the decomposition of the twisted Dirac operator. \\
\noindent
For $\kappa,\iota \in \Lambda_\lambda$ let $r_{\iota}$ denote the restriction to the summand $\mS_{\iota}\in \mV_\lambda \otimes \mS$ and $p_\kappa$ the projection to $\mS_{\kappa}$ in the same diagram. If we then choose the normalization of the operator $\TW{\kappa}{\iota}$ in such a way that $\TW{\kappa}{\iota}=p_\kappa \circ D \circ r_\iota$, the scalar operator $p_{\kappa} \circ D^2 \circ r_\iota$ gives rise to three operator identities. First of all, for $\kappa=\iota$ we get that
\begin{equation}\label{splitlaplace}
-\LA{\kappa}=\mcR_{\kappa}^2 + \sum_{\substack{\omega \in \Lambda_\lambda \\ |\kappa,\omega|=1}} \TW{\kappa}{\omega} \TW{\omega}{\kappa}\ .
\end{equation}
There are precisely $n$ terms in the sum, but some of them can be zero (when $\omega$ is not dominant). Secondly, for $|\kappa,\iota|=1$ we get
\begin{equation}\label{moverarita}
\TW{\kappa}{\iota}\mcR_{\iota} + \mcR_{\kappa}\TW{\kappa}{\iota} = 0\ .
\end{equation}
Finally, for $|\kappa,\iota|=2$ we have that
\begin{equation}\label{moveturn}
\TW{\kappa}{\theta}\TW{\theta}{\iota} + \TW{\kappa}{\omega}\TW{\omega}{\iota} = 0\ ,
\end{equation}
with $\theta$ and $\omega$ the only two weights in $\Lambda_\lambda$ on a distance 1 from both $\kappa$ and $\iota$. Summands with $|\kappa,\iota|>2$ do not give further relations.\\
\noindent
Let us consider a specific path $(\lambda_i)_0^N$ between $\mu$ and $\nu$, $\mu \succ \nu$, going through the weights:
\begin{align*}
\lambda_0&=(\nu_1,\nu_2,\nu_3,\ldots,\nu_{n-1},\nu_n) \\
\lambda_{i_1}&=(\mu_1,\nu_2,\nu_3,\ldots,\nu_{n-1},\nu_n) \\
\lambda_{i_2}&=(\mu_1,\mu_2,\nu_3,\ldots,\nu_{n-1},\nu_n) \\
&\vdots \\
\lambda_{i_{n-1}}&=(\mu_1,\mu_2,\mu_3,\ldots,\mu_{n-1},\nu_n) \\
\lambda_{N}&=(\mu_1,\mu_2,\mu_3,\ldots,\mu_{n-1},\mu_n),
\end{align*}
where $0\equiv i_0 \leq i_1\leq i_2 \leq \ldots i_{n-1} \leq i_n \equiv N$. These weights determine the path uniquely, since between any successive pair of them there is a unique path. The sequence $\{Ch(i)\}_1^N$ is non-decreasing and increases at each $i_k$, $k \in \{1,\ldots,n-1\}$. For a path $(\omega_i)_0^N$ from $\nu$ to $\mu$ different from the given path $(\lambda_i)_0^N$, there is always an index $j \in \{1,\ldots,N\}$ such that $Ch(j)<Ch(j-1)$. The path $(\omega_i)_0^N$ can then be deformed, by changing 
$$\omega_{j-1}=\omega_{j-2}+\eps_{Ch(j-1)}=\omega_{j}-\eps_{Ch(j)}$$ 
into a new weight $\theta$, given by
$$\theta:=\omega_{j-2}+\eps_{Ch(j)}\equiv \omega_{j}-\eps_{Ch(j-1)}\ .$$
Using operator identity (\ref{moveturn}), we then have that
\begin{equation*}
\TW{\omega_{N}}{\omega_{N-1}} \TW{\omega_{N-1}}{\omega_{N-2}} \ldots \TW{\omega_{j}}{\omega_{j-1}}\TW{\omega_{j-1}}{\omega_{j-2}}\ldots \TW{\omega_{1}}{\omega_{0}}
\end{equation*}
equals
\begin{equation*}
-
\TW{\omega_{N}}{\omega_{N-1}} \TW{\omega_{N-1}}{\omega_{N-2}} \ldots \TW{\omega_{j}}{\theta}\TW{\theta}{\omega_{j-2}}\ldots \TW{\omega_{1}}{\omega_{0}}\ .
\end{equation*}
In such a way it is clearly possible to deform any path from $\nu$ to $\mu$ into $(\lambda_i)_0^N$ and the resulting path operators differ only by a multiple coming from the signs and the way how different twistor operators are normalized. An analogous argument works for path operators defined by a reverse path. \\
\noindent
We will now choose the normalization of the twistor operators in such a way that $\PA{\mu}{\nu}$ will no longer depend on the path. Let us therefore write a weight $\mu$ as $(\mu_1,\mu_2,\ldots,\mu_m,0,\ldots,0)$, with $\mu_m > 0$ for $m\in\{1,\ldots,n\}$. We then define $\PA{\mu}{0}$ as the composition of twistor operators along the unique path through
\begin{align*}
\lambda_0&=(0,0,0,\ldots,0) \\
\lambda_{i_1}&=(\mu_1,0,0,\ldots,0) \\
\lambda_{i_2}&=(\mu_1,\mu_2,0,\ldots,0) \\
&\vdots \\
\lambda_{N}&=(\mu_1,\ldots,\mu_m,0,\ldots,0)\ ,
\end{align*}
and choose an arbitrary normalization for $\PA{\mu}{0}$ along this preferred path. For any index $p<m$, we can now normalize the twistor operator  $\TW{\mu+\eps_p}{\mu}$ by the requirement $\PA{\mu+\eps_p}{0}=\TW{\mu+\eps_p}{\mu}\PA{\mu}{0}$. This makes the complete diagram of twistor operators commutative and also chooses the normalization for any path operator $\PA{\mu}{\nu}$. The same way we can normalize path operators for reverse paths.\\
\noindent
If $(\lambda)_0^N$ is a path from $\nu$ to $\mu$ then by the relation (\ref{moverarita}) we get that
\begin{equation*}
\mcR_{\mu}\PA{\mu}{\nu}= - \TW{\mu}{\lambda_{N-1}}\mcR_{\lambda_{N-1}}\PA{\lambda_{N-1}}{\nu},
\end{equation*}
and continuing this process we establish the second assertion up to a multiple. An arbitrary normalization of $\mcR_{0}$ then gives us a consistent choice of overall normalization of other operators $\mcR_{\mu}$ by requiring that
\begin{equation*}
\mcR_{\mu}\PA{\mu}{0}=\PA{\mu}{0}\mcR_{0}\ .
\end{equation*}
\noindent
To prove the third assertion, we need to show that for any $\lambda \prec \mu$ such that $\lambda \notin B(\mu)$, the operator $\PA{\mu}{\lambda}$ can be expressed via a path in which a composition of two specific operators is zero. If $\lambda \notin B(\mu)$ then for some $i\in\{1,\ldots,n-1\}$ one has that $\lambda_i<\mu_{i+1}$ and $\mu_{i+1}>\mu_{i+2}$ (with the additional definition $\mu_{n+1}:=0$). We choose the path from $\lambda$ to $\mu$ as a succession of paths from $\lambda$ to $\lambda_-$, then the path $(\lambda_-,\lambda_0,\lambda_+)$ and then from $\lambda_+$ to $\mu$, where
\begin{align*}
\lambda_-&=(\mu_1,\mu_2,\ldots,\mu_{i-1},\mu_{i+1}-1,\mu_{i+1}-1,\mu_{i+2},\ldots,\mu_{n-1},\mu_n) \\
\lambda_0&=(\mu_1,\mu_2,\ldots,\mu_{i-1},\mu_{i+1},\mu_{i+1}-1,\mu_{i+2},\ldots,\mu_{n-1},\mu_n) \\
\lambda_+&=(\mu_1,\mu_2,\ldots,\mu_{i-1},\mu_{i+1},\mu_{i+1},\mu_{i+2},\ldots,\mu_{n-1},\mu_n).
\end{align*}
We see that all three weights are dominant and $\lambda_0 \in B(\mu)$, since $\mu_{i} \geq \mu_{i+1} \geq \mu_{i+1}$, $\mu_{i+1} \geq \mu_{i+1}-1 \geq \mu_{i+2}$ and the other coordinates are the same as in $\mu$ itself. Similarly $\lambda_+ \in B(\mu)$. As $\lambda_-\succ \lambda$ and $\mu \succ \lambda_+$, the path operators $\PA{\lambda_-}{\lambda}$, $\PA{\mu}{\lambda_+}$ are well defined and
\begin{equation*}
\PA{\mu}{\lambda}=\PA{\mu}{\lambda_+}\TW{\lambda_+}{\lambda_0}\TW{\lambda_0}{\lambda_-}\PA{\lambda_-}{\lambda}
\end{equation*}
Now a weight
\begin{equation*}
\kappa:=(\mu_1-1,\mu_2-1,\ldots,\mu_{i-1}-1,\mu_{i+1}-1,\mu_{i+1}-1,\ldots,\mu_{n-1},\mu_n)
\end{equation*}
is still dominant and the weights $\lambda_-$, $\lambda_0$, $\lambda_+$ are sums of $\kappa$ with
\begin{align*}
\omega_-&=(1,1,\ldots,1,0,0,0,\ldots,0,0) \\
\omega_0&=(1,1,\ldots,1,1,0,0,\ldots,0,0) \\
\omega_+&=(1,1,\ldots,1,1,1,0,\ldots,0,0),
\end{align*}
respectively, i.e. the representations $\mS_{\lambda_-}$, $\mS_{\lambda_0}$, $\mS_{\lambda_+}$ are Cartan products of $\mS_\kappa$ with the representations $\mV_{\omega_-}$, $\mV_{\omega_0}$, $\mV_{\omega_+}$. The sequence of first order invariant differential operators
\begin{equation*}
\mcV_{\omega_-} \rightarrow \mcV_{\omega_0} \rightarrow \mcV_{\omega_+}
\end{equation*}
is in fact part of the de Rham complex and so the composition is zero. If we twist this sequence by $\mS_\kappa$, the composition operator is still zero and so is its restriction and projection to the Cartan parts $\mcS_{\lambda_-}$ and $\mcS_{\lambda_+}$ respectively. Since the only weight $\theta \in \Lambda_d$ such that $|\theta,\lambda_+|=1$ and $|\theta,\lambda_-|=1$ is $\lambda_0$, the twisted de Rham operator restricted on $\mcS_{\lambda_-}$ an projected to $\mcS_{\lambda_+}$ must be a composition of twistor operators
\begin{equation*}
\mcS_{\lambda_-} \rightarrow \mcS_{\lambda_0} \rightarrow \mcS_{\lambda_+},
\end{equation*}
which is up to a multiple $\TW{\lambda_+}{\lambda_0}\TW{\lambda_0}{\lambda_-}$. Hence $\PA{\mu}{\lambda}=0$. A similar argument shows that $\PA{\lambda}{\mu}=0$.
\end{proof}
\begin{theorem}
Let $\mu \in \Lambda_d$, then for $n>\mu_1$
\begin{equation*}
\LA{\mu}^n=\mcR_{\mu}\left[\sum_{\lambda \in B(\mu)} c(\mu,\lambda) \PA{\mu}{\lambda} \LA{\lambda}^{n-|\mu,\lambda|-1} \PA{\lambda}{\mu} \right]\mcR_{\mu},
\end{equation*}
where $c(\mu,\lambda)$ are constants.
\end{theorem}
\begin{proof}
Relation (\ref{splitlaplace}) for the twisted Dirac operator on $\mV_\mu\otimes\mS$ gives:
\begin{equation*}
\LA{\mu}^n=c(\mu,\mu)\mcR_{\mu}\LA{\mu}^{n-1}\mcR_{\mu} + \sum_{\substack{\lambda \prec \mu \\ |\mu,\lambda|=1}} c_1(\mu,\lambda) \TW{\mu}{\lambda}\LA{\lambda}^{n-1} \TW{\lambda}{\mu},
\end{equation*}
where we have to include constants $c(\mu,\mu)$ and $c_1(\mu,\lambda)$ to compensate for our choice of normalization of the RS and twistor operators, and for the sign included in relation (\ref{splitlaplace}). If we expand the Laplace operators in the terms of the summation over $\lambda$ in a similar way, we get:
\begin{align*}
\LA{\mu}^n&=c(\mu,\mu)\mcR_{\mu}\LA{\mu}^{n-1}\mcR_{\mu} \\ &+ \sum_{\substack{\lambda \prec \mu \\ |\mu,\lambda|=1}} c(\mu,\lambda) \PA{\mu}{\lambda}\mcR_{\lambda}\LA{\lambda}^{n-2}\mcR_{\lambda} \PA{\lambda}{\mu} +  \sum_{\substack{\lambda \prec \mu \\ |\mu,\lambda|=2}} c_2(\mu,\lambda) \PA{\mu}{\lambda}\LA{\lambda}^{n-2} \PA{\lambda}{\mu}
\end{align*}
We can now use the second part of Lemma \ref{pathops} to switch the HSD and path operators in the first sum and we can then expand the Laplace operators in the second sum again. If we continue plugging in and expanding, we finally arrive at
\begin{equation*}
\LA{\mu}^n=\mcR_{\mu}\left[\sum_{\substack{\lambda \prec \mu \\ |\mu,\lambda| < n}} c(\mu,\lambda) \PA{\mu}{\lambda} \LA{\lambda}^{n-|\mu,\lambda|-1} \PA{\lambda}{\mu} \right]\mcR_{\mu} + 
\sum_{\substack{\lambda \prec \mu \\ |\mu,\lambda| = n}} c_n(\mu,\lambda) \PA{\mu}{\lambda} \PA{\lambda}{\mu}\ .
\end{equation*}
The second sum is zero since the element $(\mu_2,\mu_3,\ldots,\mu_n,0)$, which lies at maximal distance from $\mu \in B(\mu)$, has distance $\mu_1$ and $n > \mu_1$. A similar argument allows us to  keep only the terms corresponding to $\lambda \in B(\mu)$ in the first sum.
\end{proof}
\noindent
As a result, we have the following conclusion, which generalizes the fact that all monogenic functions in the kernel of $\up_x$ are harmonic: 
\begin{corollary}
Given a dominant integral weight $\mu \in \Lambda_d$ and a function $f(\ux) \in \mcS_\lambda$, we have that  
$$\mcR_\lambda f(\ux) = 0\ \Rightarrow\ \Delta^{\lambda_1 + 1}f(\ux) = 0\ .$$
\end{corollary}
\noindent
Let us recall that for the cases of $\lambda=(k)$ and $\lambda=(k,l)$, where the kernel of $\mcR_\lambda$ is known explicitly, such polyharmonicity result is a direct consequence of the decomposition of the kernel into a direct sum of special solutions for ``simpler'' HSD operators $\mcR_\mu$. The set of highest weights  $\mu$ that take part in this decomposition is precisely $B(\lambda)$, i.e. it is the same set that appears in the corresponding factorization of the operator $\Delta_\lambda$. This suggests how the structure of the kernel of $\mcR_\lambda$ may look like in general, and indeed, dimension checking in randomly chosen cases supports this conjecture. We will treat this in future work.

\bibliography{clifford}

\begin{thebibliography}{10}

\bibitem{Brackx1982}
F.~Brackx, R.~Delanghe, and V.~Sou\v{c}ek.
\newblock {\em {C}lifford {A}nalysis}.
\newblock Number~76 in Research Notes in Mathematics. Pitman, London, 1982.

\bibitem{Brackx}
F.~Brackx, D.~Eelbode, and L.~Van~de Voorde.
\newblock The polynomial null solutions of a higher spin {D}irac operator in
  two vector variables.

\bibitem{Burevs2000}
J.~Bure\v{s}.
\newblock The {R}arita-{S}chwinger operator and spherical monogenic forms.
\newblock {\em Complex Variables Theory Appl.}, 43(1):77--108, 2000.

\bibitem{Burevs2001}
J.~Bure\v{s}, F.~Sommen, V.~Sou\v{c}ek, and P.~Van~Lancker.
\newblock Rarita-{S}chwinger type operators in {C}lifford analysis.
\newblock {\em Journal of Functional Analysis}, 185:425--456, 2001.

\bibitem{Burevs2001a}
J.~Bure\v{s}, F.~Sommen, V.~Sou\v{c}ek, and P.~Van~Lancker.
\newblock Symmetric analogues of {R}arita-{S}chwinger equations.
\newblock {\em Annals of Global Analysis Geometry}, 21(3):215--240, 2001.

\bibitem{DeSchepper2010}
H.~De~Schepper, D.~Eelbode, and T.~Raeymaekers.
\newblock On a special type of solutions of arbitrary higher spin {D}irac
  operators.
\newblock {\em Journal of Physics A: Mathematical and Theoretical},
  43(32):325208, 2010.

\bibitem{Delanghe1992}
R.~Delanghe, F.~Sommen, and V.~Sou\v{c}ek.
\newblock {\em Clifford analysis and spinor valued functions}.
\newblock Kluwer Academic Publishers, Dordrecht, 1992.

\bibitem{Eelbode2009a}
D.~Eelbode and D.~\v{S}m\'{\i}d.
\newblock Algebra of invariants for the {R}arita-{S}chwinger operators.
\newblock {\em Annales Academiae Scientiarum Fennicae Mathematica},
  34:637--649, 2009.

\bibitem{Eelbode2009}
D.~Eelbode and D.~\v{S}m\'{\i}d.
\newblock {\em Polynomial invariants for the {R}arita-{S}chwinger operator},
  pages 125--135.
\newblock Birkh\"{a}user Basel, Basel, 2009.

\bibitem{Rarita1941}
W.~Rarita and J.~Schwinger.
\newblock On a theory of particles with half-integral spin.
\newblock {\em Phys. Rev.}, 60(1):61, Jul 1941.

\bibitem{VanLancker2006}
P.~Van~Lancker.
\newblock Rarita-{S}chwinger fields in the half space.
\newblock {\em Complex Variables and Elliptic Equations}, 51(5-6):563--579,
  2006.

\end{thebibliography}
\end{document}